\documentclass[12pt]{amsart}
\usepackage[utf8]{inputenc}
\usepackage[margin=1in]{geometry}
\usepackage{amscd}
\usepackage[pdftex,pdfborder={0 0 0}]{hyperref}
\usepackage{fontenc}
\usepackage{textcomp}
\usepackage{comment}
\usepackage{enumitem}
\usepackage{emptypage}
\usepackage[all]{xy}

\usepackage{amssymb,amsmath,amsthm}
\usepackage{amsrefs}        
\usepackage{mathrsfs}
\usepackage{xcolor}
\usepackage[all]{xy}
\usepackage{cancel}
\usepackage{graphicx}
\usepackage{lscape}
\usepackage{cite}
\usepackage{afterpage}

\newtheorem{defi}{Definition}[section]
\newtheorem{thm}{Theorem}

\newtheorem{prop}{Proposition}

\newtheorem{exe}{Example}
\newtheorem{lem}{Lemma}

\newtheorem{rema}{Remark}

\newtheorem{coro}{Corollary}

\newcommand\nn{\mathbb{N}}

\newcommand\zz{\mathbb{Z}}
\newcommand\rr{\mathbb{R}}

\newcommand\ttt{\mathbb{T}}
\newcommand\hh{\mathbb{H}}

\newcommand\ungra{\mathbf{1}}

\newcommand\fra[2]{\displaystyle\frac{#1}{#2}}

\newcommand\tq{\mbox{ } | \mbox{ }}

\newcommand\cali[1]{\mathcal{#1}}
\newcommand*\diff{\mathop{}\!\mathrm{d}}

\DeclareMathOperator{\sln}{SL}

\DeclareMathOperator{\supp}{supp}

\DeclareMathOperator{\Stab}{Stab}

\author[A. Pinochet Lobos]{Antoine Pinochet Lobos}\thanks{I2M, CNRS UMR7373, Université d'Aix-Marseille, a.p.lobos \emph{at} outlook.com}
\author[C. Pittet]{Christophe Pittet}\thanks{I2M, CNRS UMR7373, Université d'Aix-Marseille et Section de Mathématiques, Faculté des Sciences, Université de Genève, pittet \emph{at} math.cnrs.fr\\
The authors acknowledge support of the FNS grant 200020-200400.}

\title[Universal lower bounds  for the discrepancies of actions of  a l.c. group]{Universal lower bounds for the discrepancies of actions of a locally compact group}
\date{}

\keywords{Atomless measure, convolution operator, discrepancy, Koopman representation, locally compact group, measure-preserving action, regular representation, spectral gap,  unitary representation, von Neumann ergodic theorem, weak containment.}

\subjclass[2020]{Primary: 37A15 ; Secondary: 37A30}

\begin{document}

\begin{abstract}
We prove universal lower bounds  for discrepancies (i.e. sizes of spectral gaps of averaging operators) of measure-preserving actions of  a locally compact group on probability spaces. For example,  a locally compact Hausdorff unimodular group $G$, acting continuously, by measure-preserving transformations, on a compact atomless probability space $(X,\nu)$, with an orbit $Gx_0$ of measure zero, contained in the support of $\nu$, and with compact stabilizer (i.e. $G_{x_0}$ is compact) has the following property:  any finite positive regular Borel measure $\mu$ on $G$ satisfies $$\|\pi_0(\mu)\|_{2\to 2}\geq \|\lambda_G(\mu)\|_{2\to 2},$$ where $\pi_0$ denotes the Koopman representation of $G$, defined by the given action, and $\lambda_G$ denotes the left-regular representation of $G$. The lower bounds we prove generalize the universal lower bounds for the discrepancies of measure-preserving actions of a discrete group. Many examples show that the generalization from discrete groups to locally compact groups requires some additional hypothesis on the action (we detail some examples of actions of amenable groups with a spectral gap, due to Margulis). Well-known examples and results of Kazhdan and Zimmer show that the discrepancies of some actions of Lie groups on homogeneous spaces match exactly the universal lower bounds we prove. 
\end{abstract}
\maketitle
\tableofcontents

\section{Introduction}
The discrepancies of a measure-preserving group action are related to the possible rates of convergence in ergodic von Neumann type theorems.
They measure the sizes of spectral gaps of averaging operators. See Definition \ref{defidiscrep} and Proposition \ref{prop: Koopman} below.
The notion of discrepancy can be interpreted as a measure of the ergodicity of the action \cite[Thm 4.2]{ShaTata}.

Isometric group actions, on the $2$-dimensional sphere, of free groups with the smallest possible discrepancies (equivalently, with the fastest possible rate of convergence) have been constructed by Lubotzky, Phillips, and Sarnak. See \cite{LPS} and \cite{PinPitExactRate}.
For discrete groups, there exist universal lower bounds for the  discrepancies of  measure-preserving actions on atomless probability spaces, which can be expressed with the help of the regular representation. In other words, there are lower bounds for the discrepancies, which are valid for any measure-preserving action of the group on an atomless probability space, and these universal lower bounds, which are inherent to the structure of the group, can be extracted from the regular representation of the group. The above mentioned actions considered by Lubotzky, Phillips, and Sarnak, have discrepancies matching exactly these universal lower bounds. See \cite{PinPitExactRate}.

A remarkable  and well-known connection between dynamical systems and the theory of unitary representations is that discrepancies of locally compact group actions are equal to the norms of operators associated to the corresponding Koopman representations (see  Proposition \ref{prop: Koopman}  below for the precise statement we will use). In other words, the discrepancies we consider (as initially defined in \cite[0.3 and 0.3']{LPS}) measure the sizes of the spectral gaps of Koopman operators.  These norms, equivalently these spectral gaps, have been studied by many authors (see for example \cite{ShaTata} \cite{FurSha}, \cite{BekGui}, \cite{GorNev}, \cite{GorNevBook}). Spectral transfer, in the sense of Nevo \cite{Nev}, is a powerful method to estimate these norms. Roughly speaking, the idea is to apply different unitary representations, let's say $\pi$ and $\lambda$, of a locally compact group $G$, to a given probability measure $\mu$ on $G$, and to compare the operator norms of the operators $\pi(\mu)$  and $\lambda(\mu)$ acting on different Hilbert spaces. See also \cite[4. Spectral gaps and ergodic theorems]{GorNev}. For example, one of the main point in the work of Lubotzky, Phillips, Sarnak, cited above, is the equality between the spectrum of $\pi_0(\mu_n)$ and the one of $\lambda_G(\mu_n)$, where $\pi_0$ is a specific (defined by arithmetic conditions) Koopman representation on the $2$-dimensional sphere of a finitely generated free group $G$, where $\lambda_G$ is the left regular representation of $G$, and where $\mu_n$ is the uniform measure on the sphere of radius $n$ with center the identity $e\in G$ with respect to the word metric defined by the canonical generating set of $G$. When $n=1$, the equality of the spectra is proved with the help of Deligne's solution to the Weil conjecture (see \cite[page 419]{LPS2} and \cite{Lub}), the cases with $n>1$ follow from the case $n=1$ and from the spectral theorem applied to Hecke elements. Exact computations can be performed with the help of the Harish-Chandra function of the natural boundary representation of the free group (see \cite{PinPitExactRate}).  In \cite[Theorem 1]{Nev}, Nevo applies the spectral transfer method to any Kazhdan connected semi-simple Lie group $G$ with finite center and without compact factor. Relying on estimates of some coefficients of unitary representations of $G$, due to Cowling \cite{Cow}, Nevo shows that there is an even integer $n(G)$, such that the Koopman representation $\pi_0$ of $G$, defined by any measure-preserving ergodic action of $G$ on a probability space, satisfies, for any Borel probability measure $\mu$ on $G$,
$\|\pi_0(\mu)\|_{2\to 2}\leq\|\lambda_G(\mu)\|_{2\to 2}^{1/n(G)}$. In some special cases, it is known (see \cite[4.1]{Nev} and  Proposition \ref{exemple optimal} below) that the upper bound on $\|\pi_0(\mu)\|$ (we drop the $2\to 2$ indexes) can be strengthened to $\|\pi_0(\mu)\|\leq\|\lambda_G(\mu)\|$. 

Most of the results on discrepancies are upper bounds (one tries to get the fastest possible rate of convergence, by proving $\|\pi_0(\mu)\|<1$ is as small as possible, equivalently by proving the size of the spectral gap is as large as possible). But lower bounds have also been studied, for example in \cite[Th\'eor\`eme 3.3]{Sev}, \cite[Theorem 2]{Clo}, \cite[Theorem 4]{Pis}, \cite[Theorem 1.3]{LPS}. All these lower bounds concern discrete groups. Proposition \ref{mainthm discrepance localement compact} below generalizes all these lower bounds, and  applies to a general locally compact Hausdorff group $G$. Its conclusion is the lower bound 
\begin{align}\label{lower bound}
	\|\pi_0(\mu)\|\geq\|\lambda_G(\mu)\|,
\end{align}
for any positive bounded Borel regular measure $\mu$ on $G$, where $\pi_0$ is the Koopman representation defined by a measure-preserving $G$-action on an atomless probability space, satisfying technical hypothesis we will comment on.
The proof of Proposition \ref{mainthm discrepance localement compact} is inspired from the proof of \cite[Theorem 3.6]{PinPitExactRate}, which in turn is based on  the idea of the proof of \cite[Th\'eor\`eme 3.3]{Sev}. In contrast with the case of discrete groups, Inequality (\ref{lower bound}) can't be true for general measure preserving actions of general locally compact Hausdorff groups on an atomless probability space: there are many examples of amenable Lie groups acting by measure preserving transformations on atomless probability spaces which have a spectral gap (see Section \ref{contreex}). Condition (3) of Proposition \ref{mainthm discrepance localement compact} (which is trivially fulfilled in the case $G$ is discrete) is an additional assumption on the action which makes possible to handle the general case of a locally compact Hausdorff group $G$. Heuristically, Condition (3) guarantees that F{\o}lner quotients of partial orbits growth sub-linearly (see Subsection \ref{heuristic} below for an example and some pictures). Condition (3) potentially excludes some  interesting actions, but we show in (the proof of) Theorem \ref{thm: orbites negligeables} that it is fulfilled in many natural situations. 

Shalom, and independently, Dudko and Grigorchuk, have proven that for any discrete group acting by measure-preserving transformations on any atomless probability space, the associated   Koopman representation $\pi_0$ (i.e. the one on square integrable functions of zero integral, which is the one we study in this paper)  weakly contains the quasi-regular representation associated with almost every point stabilizer; see \cite[Proposition 7]{DudGri} (where the atomless probability space is assumed to be standard) and \cite{ShaTata} (where the atomless probability space is not assumed to be standard but the result is stated without proof). This weak containment result implies Inequality (\ref{lower bound}) for all positive finite regular Borel measure $\mu$: one applies the characterization of weak containment explained in \cite[Remark 8.B.6 (2) page 244]{BekDelaHa}, and an inequality due to Shalom \cite[Lemma 2.3]{ShaAnnals} (which requires the measure $\mu$ to be positive). The same conclusion, namely Inequality (\ref{lower bound}) for all positive finite regular Borel measure $\mu$, is obtained directly by applying Proposition \ref{mainthm discrepance localement compact}. In the case there exists a trivial stabilizer, Inequality (\ref{lower bound})  is true for any bounded Borel regular measure $\mu$, not only for positive $\mu$; equivalently $\pi_0\succ \lambda_G$.  

In the case $G$ is totally discontinuous and second countable, Caprace, Kalantar, and Monod, prove an interesting weak containment result, for some measurable actions, preserving \emph{the class} of the measure, on a standard probability space (see \cite[Theorem F]{CapraceKalantarMonod}) but it seems that no generalization  of the weak containment result of Shalom and Dudko-Grigorchuk mentioned above is known for locally compact Hausdorff groups. Notice that \cite[Theorem F]{CapraceKalantarMonod} says nothing about Inequality (\ref{lower bound}): take $G$  a compact group with Haar measure $\mu_G$, for example $G=\mathbb Z_p$ the $p$-adic integers, and $\pi_0$ the Koopman representation on $L^2_0(G,\mu_G)$, then $\pi_0(\mu_G)=0$, whereas the amenability of $G$ implies $\|\lambda_G(\mu_G)\|=\mu_G(G)$.

In the proof of Theorem \ref{thm: orbites negligeables}, we use a technic from geometric group theory (volume approximation in a locally compact group based on nets relative to a word metric) to show that if $G$ is locally compact Hausdorff unimodular and acts continuously and freely (more generally with compact stabilizers) with an orbit of measure zero (contained in the support of the measure), then all the hypothesis of Proposition \ref{mainthm discrepance localement compact} are fulfilled, hence Inequality (\ref{lower bound}) holds in this case (the existence of a orbit with zero measure is a necessary condition as many examples show: see Section \ref{contreex}).  We do not know if the compact stabilizer hypothesis and/or the unimodularity hypothesis in Theorem \ref{thm: orbites negligeables} can be removed (but the proof we give, using nets, needs both hypothesis). 

An application of Theorem \ref{thm: orbites negligeables} is the following. Suppose  $H$ is a closed unimodular subgroup of a Lie group $G$ and $\Gamma$ is a lattice in $G$ such that $\Gamma\cap H$ is finite. If the dimension of $H$ is smaller than the one of $G$, then the Koopman representation associated to the natural action of $H$ on $G/\Gamma$ satisfies
Inequality (\ref{lower bound}) for any positive bounded Borel regular measure on $H$. See Corollary \ref{coro: actions sur reseaux}. In Proposition \ref{exemple optimal} we give examples of actions of Lie groups where the lower bounds brought by Theorem \ref{thm: orbites negligeables} and Corollary \ref{coro: actions sur reseaux}
match exactly upper bounds deduced from well known results of Kazhdan and Zimmer.

\section{Outline of the paper}
In Section \ref{preliminaries} we fix notation and briefly recall well-known links between discrepancies, operator norms, and Koopman representations.
In Section \ref{statement} we state the main results of the paper.
In Section \ref{heuristic} we illustrate Condition (3) from Proposition \ref{mainthm discrepance localement compact} with an example and two pictures. 
Section \ref{demo thm} is devoted to the proofs of the results stated in Section \ref{statement}.  
In Section \ref{contreex}, we recall a result of Margulis implying the existence of plenty of natural actions of amenable groups with a spectral gap.  
Finally, we describe in Section \ref{sect: exemple de Zimmer} an action of $\mathbb{R}^2 \rtimes \sln_2(\mathbb{R})$ on a finite-volume homogeneous space, where the inequality of Corollary \ref{coro: actions sur reseaux} is sharp. 

\emph{Acknowledgements.}
We warmly thank Pierre-Emmanuel Caprace who was at the origin of the questions that we address in this paper; we are grateful to him for helpful discussions and remarks.

\section{Preliminaries}\label{preliminaries}

In this paper, $G$ denotes a locally compact Hausdorff topological group and $\mu_G$ is a left Haar measure on $G$.

\begin{defi}[Discrepancy]
\label{defidiscrep}
Let $(X,\cali{T},\nu)$ be a probability space, and $G \curvearrowright X$ a measurable action which preserves $\nu$. Let $\mu$ be a Borel probability measure on $G$. We call the number \[\delta(\mu) := \sup_{\substack{\phi \in L^2(X,\nu)\\ \Vert \phi \Vert_2 = 1}} \left\Vert x \mapsto \int_G \phi\left( g^{-1}x\right)\diff \mu - \int_X \phi \diff \nu\right\Vert_2.\] the \textbf{discrepancy} of $\mu$.
Similarly, for any continuous, non-negative, compactly-supported function $f$ on $G$ of integral $1$, we denote by $\delta(f)$ the discrepancy of the probability measure with density $f$ with respect to $\mu_G$.
\end{defi}

We denote  \[\begin{array}{rcl}
\lambda_G : G &\rightarrow &U(L^2(G,\mu_G))\\
g&\mapsto &\displaystyle \phi \mapsto \left(h \mapsto \phi(g^{-1}h)\right)\\
\end{array}\] the left regular representation of $G$.

\begin{defi}[Koopman representation]\label{defi: Koopman}

Consider an action of $G$ on a measure space $(X,\nu)$ by measurable transformations preserving the measure.

The \textbf{Koopman representation} associated to this action is defined by \[\begin{array}{rcl}
\pi : G &\rightarrow &U(L^2(X,\nu))\\
g &\mapsto &(\phi \mapsto (x \mapsto \phi(g^{-1}x)))\\
\end{array}\]
\end{defi}

\begin{rema}\label{remarks about Koopman}
\begin{enumerate}
    \item The left regular representation is the Koopman representation for the action of the locally compact Hausdorff group $G$ endowed with a left Haar measure $\mu_G$ by left multiplication on itself.
    \item For every $g \in G$, the operator $\pi(g)$ defined above is unitary and $\pi$ is a homomorphism. Moreover, if $G$ is $\sigma$-compact, locally compact Hausdorff, if $\nu$ is $\sigma$-finite and $L^2(X,\nu)$ is separable, then the Koopman representation $\pi$ is a unitary representation  (that is, it is continuous if we endow $U(L^2(X,\nu))$ with the strong operator topology; see \cite[Proposition A.6.1]{BHV}). When we consider a Koopman representation, or any unitary representation, it is always implicitly assumed that it is strongly continuous.
    \item \label{rema: defi subrep Koopman} If $\nu$ is a finite measure, then $\textbf{1}_{X}$ is a fixed vector, and its orthogonal complement is a sub-representation. This subspace is the subset of all (representatives of) functions with zero integral and we denote it by $L^2_0(X,\nu)$. We denote the associated unitary sub-representation \[\pi_0 : G \rightarrow U(L^2_0(X,\nu)).\]
We also call $\pi_0$ a Koopman representation.
\end{enumerate}
\end{rema}

Let $\pi$ be a  unitary representation of $G$ on a Hilbert space $\mathcal{H}$. Let $\mathcal{B}(\mathcal{H})$ be the involutive algebra of bounded operators on $\mathcal{H}$  and let $\mu$ be a finite regular Borel measure on $G$. Recall that the operator $\pi(\mu)$ in defined in the following way: \[\pi(\mu) := \int_G \pi(g)\diff\mu(g),\]that is, $\pi(\mu)$ is the only bounded operator on $\mathcal{H}$ such that for all $\xi,\eta \in \mathcal{H}$, \[\langle \pi(\mu)\xi,\eta\rangle := \int_G \langle \pi(g)\xi,\eta\rangle \diff\mu(g).\]
The unitary representation $\pi$ induces in this way a $*$-representation of the involutive algebra $\mathcal{M}(G)$ of regular complex Borel measures on $G$: 
\[\begin{array}{rcl}
\mathcal{M}(G) &\rightarrow &\mathcal{B}(H)\\
\mu &\mapsto &\displaystyle\pi(\mu).\\
\end{array}\]
See  \cite[Appendix 6]{BekDelaHa}. For a regular symmetric probability measure $\mu$ on $G$, the operator $\pi(\mu)$ is self-adjoint, and its operator norm is smaller or equal than $1$. 

\begin{prop}[Koopman representation and discrepancies] \label{prop: Koopman}
Consider an action of $G$ on a probability space $(X,\nu)$ by measurable transformations preserving the measure, and let $\pi_0$ the sub-representation of the Koopman representation $\pi$ on the orthogonal of the subspace of constant functions (see (\ref{rema: defi subrep Koopman}) in Remark \ref{remarks about Koopman} above).
Let $\mu$ be a bounded regular Borel measure on $G$. Then $\delta(\mu) = \Vert \pi_0(\mu)\Vert$.
\end{prop}

\begin{proof}
Let $\phi \in L^2(X,\nu)$, and let us denote \[\delta(\mu,\phi) := \left\Vert x \mapsto \int_G \phi(g^{-1}x) \diff \mu - \int_X \phi \diff \nu\right\Vert_2,\]so that \[\delta(\mu) = \sup_{\substack{\phi \in L^2(X,\nu) \\ \Vert \phi \Vert_2 = 1}} \delta(\mu,\phi).\]
Let us denote $P_1$ the orthogonal projector on the subspace $\mathbb{C}\textbf{1}_X$ of constant functions, and let $P_0 := I - P_1$. Then it is straightforward to check that $P_0$ is the orthogonal projector on $L^2_0(X)$, for every $\phi \in L^2(X,\nu)$, \[P_1(\phi) = \int_X \phi \diff \nu \cdot \textbf{1}_X,\]that \[\delta(\mu,\phi) = \delta(\mu,P_0 \phi) = \Vert \pi_0(\mu)P_0\phi \Vert_2\]and the Proposition follows.
\end{proof}

A well-known consequence  of Proposition \ref{prop: Koopman} is the following. Assume $\mu$ is a Borel regular symmetric probability measure $\mu$ on $G$.  If $\delta(\mu) <1$,  then there is a uniform exponential rate of convergence in the corresponding von Neumann's ergodic theorem; that is, we have\[\forall \phi \in L^2(X,\nu),\, \|\phi\|_2\leq 1,\,\quad \left\Vert \pi(\mu)^n \phi - \int_X \phi\diff \nu\right\Vert_2 \leq \delta(\mu)^n.\]

\section{Statement of the results}\label{statement}
Let us recall that an \emph{atom} of a measure space $(X,\nu)$ is a measurable subset $A$ such that $\nu(A) > 0$ and every measurable subset $B$ of $A$ is such that \[\nu(B) \in \{0;\nu(A)\}.\] A measure space that does not have atoms is called \emph{atomless}.

The following proposition is a crucial technical tool in this paper.

\begin{prop}\label{mainthm discrepance localement compact}
Let $(X,\cali{B},\nu)$ be an atomless probability space. Let $G$ be a locally compact Hausdorff group acting on $X$ by  measure-preserving transformations.
We assume the corresponding Koopman representation $\pi_0$ is strongly continuous.  

Assume for each compact symmetric neighborhood $S$ of the identity $e\in G$, there exist a neighborhood $F$ of $e\in G$
and a sequence $(B_n)_{n \in \mathbb{N}}$ of measurable subsets of $X$, with the following properties:
\begin{enumerate}
\item $\forall n \in \nn, \ \nu(B_n) > 0$,
\item $\displaystyle\inf \{\nu(B) \tq B \in \cali{B}, \  S^nB_n \subset B\}  < \frac{1}{2}$,
\item $\displaystyle\limsup_{n \to \infty} \left(\inf_{g \in F}\fra{\nu(gB_n\cap B_n)}{\nu(B_n )}\right)^{\frac{1}{n}} \geq 1$.
\end{enumerate}
Then,  any  finite positive regular Borel measure $\mu$ on $G$ satisfies:
$$\Vert\pi_0(\mu)\Vert\geq\Vert\lambda_G (\mu) \Vert.$$
\end{prop}

\begin{rema}
The reader should notice that in condition $(2)$ above, the set $S^nB_n$ may fail to be measurable (see for example the construction, in \cite{ERDSTONE}, of subsets $K$ and $B$ of $\mathbb{R}$ such that $K$ is compact, $B$ is a $G_\delta$ - and therefore, Borel - but $K+B$ is not Borel). 
\end{rema}

We first notice a straightforward application of the proposition (which generalizes \cite[Theorem 3.6, p. 77]{PinPitExactRate}):  if  $(X,\nu)$ is an atomless probability space, $G$ a \emph{discrete} group, $G\curvearrowright X$ a measure-preserving action, 
and $\mu$ a probability measure on $G$, then 
$$\Vert\pi_0(\mu)\Vert\geq\Vert\lambda_G (\mu) \Vert.$$
Indeed, in an atomless probability space, there exist measurable sets of arbitrary small nonzero measure (see  Theorem \ref{Sierpinski} below).
The discretness of $G$  implies that any compact
subset $S$ of $G$ is finite. Hence for any measurable set $B_n$ of $X$ of small measure, namely such that $\nu(B_n)<\frac{1}{2|S|^{n}}$, we have $\nu(S^nB_n)\leq|S|^n\nu(B_n)<\frac{1}{2}$.
The discreetness of $G$ also implies that the singleton $F=\{e\}$ is a neighborhood of $e$.  In the case $F=\{e\}$, Condition (3) is obviously true for any sequence of measurable sets of strictly positive measures. 

The conclusion of Proposition \ref{mainthm discrepance localement compact} is sharp for all countable groups. More precisely, it is  established in \cite[Corollary 3.12]{PinPitExactRate}, that for every countable group $G$, there is a measurable action (namely, the Bernoulli shift on $G$) such that for each finitely-supported probability measure $\mu$ on $G$, 
$\Vert \lambda_G(\mu) \Vert  = \delta(\mu)$.

We do not know if every locally compact Hausdorff group admits an action with a measure with discrepancies equal to the lower bounds of Proposition \ref{mainthm discrepance localement compact}.

In Proposition \ref{prop bons ouverts} below, we invoke Definition \ref{Moderate growth} and prove that Proposition \ref{mainthm discrepance localement compact} implies the following theorem.

\begin{thm}\label{thm: orbites negligeables}
Let $X$ be a topological space, $\nu$ be an atomless Borel probability measure on $X$, and $G \curvearrowright X$ a continuous action of a locally compact Hausdorff unimodular group that preserves $\nu$. We assume the corresponding Koopman representation $\pi_0$ is strongly continuous.  

Suppose there exists a point $x_0$ in the support of $\nu$ such that for any compact subset $K$ of $G$ we have $\nu(Kx_0) = 0$ and $G_{x_0}=\Stab_G(x_0)\mbox{ is compact}$.
Then, any  finite positive regular Borel measure $\mu$ on $G$ satisfies:
$$\Vert\pi_0(\mu)\Vert\geq\Vert\lambda_G (\mu) \Vert.$$
\end{thm}

The following corollary is an application of the theorem (see Subsection  \ref{proofs of the corollaries} below for details).

\begin{coro}\label{coro: actions sur reseaux} Let $G$ be a locally compact Hausdorff group, $\mu_G$ be a left Haar measure on $G$, and $H$ a unimodular closed subgroup of $G$ such that $\mu_G(H) = 0$. Let $\Gamma$ be a lattice in $G$, such that $H \cap \Gamma = \{e\}$. Let $\mu_{G/\Gamma}$ be the unique $G$-invariant Borel regular probability measure on $G/\Gamma$. We consider the action $H \curvearrowright G/\Gamma$, that preserves $\mu_{G/\Gamma}$.
Then, any  finite positive regular Borel measure $\mu$ on $H$ satisfies:
$$\Vert\pi_0(\mu)\Vert\geq\Vert\lambda_H (\mu) \Vert.$$
\end{coro}

Recall that if $G$ is amenable and $\mu$ is a regular  symmetric probability measure, whose support generates $G$, then  $\Vert \lambda_G(\mu) \Vert  = 1,$  (see \cite{Kes} for the case $G$ is discrete and \cite{PIER} for the general case). Hence,  if $G$ is discrete and amenable, $(X,\nu)$ is an atomless probability space and $G \curvearrowright X$ is a measurable action which preserves $\nu$, then the discrepancy of any symmetric probability measure on $G$ equals $1$ (see {\cite[Corollary 3.10]{PinPitExactRate}}). In the same vein, since $\mathbb{R}$ is unimodular and amenable, Theorem \ref{thm: orbites negligeables} implies that for every continuous $\mathbb{R}$-action by measure-preserving transformations with negligible orbits, the discrepancy of any compactly-supported, nonnegative, continuous function is at least $1$. For example, this yields a proof of the following folklore statement (although the result is known to experts in the field, we have not been able to localize a published proof): let $S$ be a compact, hyperbolic surface. Consider $g_t : T_1S \rightarrow T_1S$ the geodesic flow on the unit tangent bundle, and let $\nu$ be the Liouville measure on $T_1S$. Then, for every $T \in \rr^*$, \[\sup_{\substack{\phi \in L^2(T^1S,\nu)\\ \Vert \phi \Vert_2 = 1}} \left\Vert \left(x \longrightarrow \frac{1}{T}\int^T_0 \phi(g_tx) \diff t - \int_{T_1S} \phi \diff \nu\right)\right\Vert_2 = 1.\]
The same proof applies to the geodesic flow of a non-compact locally symmetric space.

There are non-amenable examples where the inequality in Corollary \ref{coro: actions sur reseaux} is an equality, and we describe one of them in Section \ref{sect: exemple de Zimmer}.

\section{Interpretation of the technical condition in Proposition \ref{mainthm discrepance localement compact}}\label{heuristic}

We give a heuristic example to illustrate the technical condition, i.e. Condition (3), in Proposition \ref{mainthm discrepance localement compact} that we postulate in order to be able to treat actions of locally compact non-discrete groups.

\begin{exe} Let us consider the action of $\rr$ on $\ttt^2 := (\rr/\zz)^2$ by translations along a line $d$ with irrational slope; that is, let $\alpha \in \rr\setminus\mathbb Q$, and define, for all $t \in \rr$, $(a,b) \in \ttt^2$, \[(a+t,b + \alpha t).\] Let $F := [-1,1]$.
On Figure 1, we represent $\mathbb{R}^2$ twice, and on each of the two pictures, we identify $\mathbb{T}^2$ with $[0,1]^2$.
Let us consider two isometric long rectangles, centered at the origin, in $\rr^2$, such that two opposite sides are parallel to $d$. Notice that the two opposite sides parallel to $d$ are short in the case of the red rectangle and are long in the case of the blue rectangle.
Consider conditions $(1),(2),(3),$ from Proposition \ref{mainthm discrepance localement compact}. Let $g$ be a small translation (i.e. with amplitude in $F$) along the line $d$ (drawn in black with a small arrow suggesting the action of $g$). Let $R_{red}$ be  the projection of the red rectangle, and $R_{blue}$ be  the projection of the blue rectangle. These two projected rectangles represent possible choices of $B_n's$ in the statement of Proposition \ref{mainthm discrepance localement compact}. One can graphically expect that $gR_{red} \cap R_{red}$ is quite small, whereas $gR_{blue} \cap R_{blue}$ is approximately as big as $R_{blue}$. One can therefore think of sequences $(B_n)_n$ that satisfy  Condition (3) in Proposition 2 as subsets that look like $R_{blue}$.
\end{exe}

\begin{figure}[h!]
\label{fig croissance moderee}
\centerline{
\includegraphics[scale=0.5]{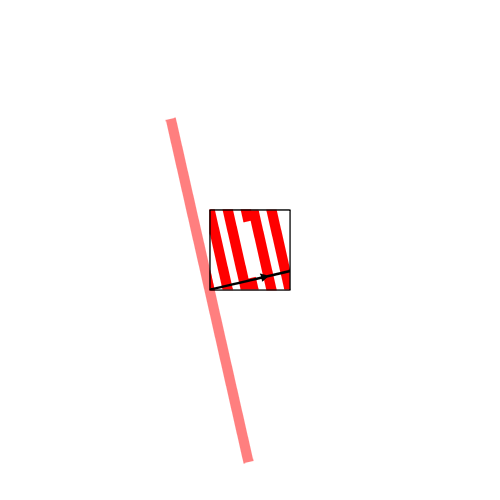}
\includegraphics[scale=0.5]{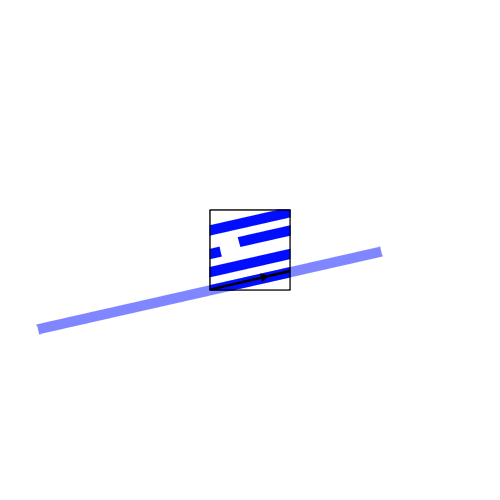}
}
\caption{The opaque-coloured rectangles are the projections of the transparent ones. The segment in black is a subset of $d$.}
\end{figure}

\section{Proofs}

\label{demo thm}

We first recall a theorem of Sierpiński on atomless probability spaces we will need in the proof of Proposition \ref{mainthm discrepance localement compact}.

\begin{thm}[See \cite{Sierpinski}]\label{Sierpinski}

Let $(X,\mathcal{B},\nu)$ be an atomless probability space. Then $(X,\nu)$ has the \emph{intermediate value property}, that is: for all measurable subsets $A,B$ of $X$ such that $A \subset B$, for all $c \in [\nu(A),\nu(B)]$, there exists a measurable subset $C$ of $X$ such that \begin{itemize}
    \item $A \subset C \subset B$;
    \item $\nu(C) = c$.
\end{itemize}
\end{thm}

(A stronger statement can be proved using Zorn's lemma: for all such $A,B$, there exists a map $f : [\nu(A),\nu(B)] \rightarrow \mathcal{B}$ such that $f$ is $(\leq,\subset)$-order-preserving and for all $c \in [\nu(A),\nu(B)]$, $\nu(f(c)) = c$.)

\subsection{Proof of Proposition \ref{mainthm discrepance localement compact}}

\label{grosthm}

We first state the following definition.

\begin{defi}[Moderate growth]\label{Moderate growth}
Let $G$ be any group, $(X,\cali{B},\nu)$ an atomless probability space, and $G \curvearrowright X$ a measure-preserving action. Let $F,S \subset G$ be measurable subsets.
We say that a sequence $(B_n)_{n \in \nn}$ of measurable subsets of $X$ is of $(S,F)$-\textbf{moderate growth} if 
\begin{enumerate}
\item $\forall n \in \nn, \ \nu(B_n) > 0$, 
\item $\displaystyle\inf \{\nu(B) \tq B \in \cali{B}, \  S^n B_n \subset B\}  < \frac{1}{2}$,
\item $\displaystyle\limsup_{n \to \infty} \left(\inf_{g \in F}\fra{\nu(gB_n\cap B_n)}{\nu(B_n )}\right)^{\frac{1}{n}} \geq 1$.
\end{enumerate}
\end{defi}

Proposition \ref{mainthm discrepance localement compact} therefore states that if for any compact symmetric neighborhood $S$ of $e$ in $G$, there exists a  neighborhood $F$ of $e$ in $G$ and an $(S,F)$-moderate growth sequence, then Inequality (\ref{lower bound}) holds.

\begin{proof}[Proof of Proposition \ref{mainthm discrepance localement compact}]
We prove the proposition in several steps.

Step 1: we assume \emph{$\mu$  is of positive type, with support $S$ a compact symmetric neighborhood of the identity $e\in G$}. As $G$ is locally compact, we may assume the neighborhood $F$ of $e$ given by the hypothesis of the proposition is compact. Let $(B ^+_n)_{n\in \nn}$ be a sequence of $F$-moderate growth.
Atomless probability spaces satisfy the so-called \emph{intermediate value property} (see Theorem \ref{Sierpinski} above).
Hence, there exists a subset $B^-_n$ that has the same $\nu$-measure as $B_n$, and such that 
\begin{equation}\label{b+b- disjoint}
    S^n B_n \cap B^-_n = \emptyset,
\end{equation}
and as $S$ is symmetric the same equality also holds with $B_n$ and $B^-_n$ exchanged.
Let \[\phi := \frac{\ungra_{B_n} - \ungra_{B^-_n}}{\Vert \ungra_{B_n} - \ungra_{B^-_n}\Vert_2}.\] The function $\phi$ is obviously in $L^2_0(X,\nu)$.
Let $\mu^{*n}$ be the $n$-fold convolution of $\mu$ with itself. Notice that 
$$S^n=(\mbox{support}(\mu))^n=\mbox{support}(\mu^{*n}).$$ 
We therefore have,
\[\begin{array}{rcl}
\Vert \pi_0(\mu^{*n}) \Vert &\geq &\langle \pi_0(\mu^{*n})\phi,\phi\rangle\\
&= &\displaystyle\frac{1}{\nu(B_n)+\nu(B^-_n)}\int_{G}\left[\nu\left(g B_n\cap B_n\right)+\nu\left(g B^-_n\cap B^-_n\right)\right] \diff \mu^{*n}(g)\\
&&- \displaystyle\frac{1}{\nu(B_n)+\nu(B^-_n)}\int_{G}\left[\nu\left(g B_n\cap B^-_n\right)+\nu\left(g B^-_n\cap B_n\right)\right] \diff \mu^{*n}(g)\\
&\stackrel{\eqref{b+b- disjoint}}{=} &\displaystyle\frac{1}{\nu(B_n)+\nu(B^-_n)}\int_{G}\left[\nu\left(g B_n\cap B_n\right)+\nu\left(g B^-_n\cap B^-_n\right)\right] \diff \mu^{*n}(g)\\
&\geq &\displaystyle\frac{1}{\nu(B_n)+\nu(B^-_n)}\int_{F}\left[\nu\left(g B_n\cap B_n\right)+\nu\left(g B^-_n\cap B^-_n\right)\right] \diff \mu^{*n}(g)\\
&\geq &\displaystyle\frac{1}{\nu(B_n)+\nu(B^-_n)} \inf_{g \in F} \nu(gB_n \cap B_n) \int_F \diff \mu^{*n}(g)\\
&= &\displaystyle \frac{1}{2}\mu^{*n}(F)\inf_{g \in F}\displaystyle\frac{ \nu(gB_n \cap B_n)}{\nu(B_n)}.\\
\end{array}\]
We know that \[\Vert \pi_0(\mu) \Vert = \limsup_{n \to \infty} \Vert \pi_0(\mu)^n\Vert^{\frac{1}{n}}.\] As $\mu$ is symmetric and $F$ is a compact neighborhood of the identity, it follows from \cite{Berg-Christensen}, that 
\[\limsup_{n \to \infty} (\mu^{*n}(F))^{\frac{1}{n}} = \Vert \lambda_G(\mu) \Vert.\]
Furthermore, the hypothesis that $\mu$ is of positive type implies that the above lim sup is in fact a genuine limit \cite{Berg-Christensen}.
We can now finish the proof of Step 1: \[\begin{array}{rcl}
\Vert \pi_0(\mu) \Vert &= &\limsup_{n \to \infty} \Vert \pi_0(\mu)^n\Vert^{\frac{1}{n}}\\
&\geq &\limsup_{n \to \infty} \frac{1}{2}^\frac{1}{n} \mu^{*n}(F)^\frac{1}{n} \left(\inf_{g \in F}\displaystyle\frac{ \nu(gB_n \cap B_n)}{\nu(B_n)}\right)^\frac{1}{n}\\
&=&\Vert \lambda_G (\mu) \Vert \limsup_{n \to \infty} \left(\inf_{g \in F}\displaystyle\frac{ \nu(gB_n \cap B_n)}{\nu(B_n)}\right)^\frac{1}{n}\\
&\geq &\Vert \lambda_G (\mu) \Vert.
\end{array}\]

Step 2: we assume \emph{the  support of $\mu$ is a compact  neighborhood of the identity}. The measure $\eta=\mu*\mu^*$ satisfies the hypothesis of Step 1. For any unitary representation $\alpha$ of $G$, $\|\alpha(\eta)\|=\|\alpha(\mu)\|^2$. Hence Step 2 follows from Step 1. 

Step 3: we assume \emph{the  support of $\mu$ is a neighborhood of the identity}. As $G$ is locally compact we may choose $K_1$, a compact neighborhood of the identity, included in the support of $\mu$. As $\mu$ is regular and finite, we may choose a increasing sequence  $K_n\subset K_{n+1}$ of compact subsets of $G$ such that $\mu(K_n)$ converges to $\mu(G)$ as $n\to\infty$, and approximate $\mu$ with the truncations $\mu_n={\bold 1}_{K_n}\cdot\mu$.  As $\mu$ and $\mu_n$ are finite regular we may apply \cite[6. Opérateurs de Radon]{Godement} to get, for any  unitary representation $\alpha$ of $G$,
\[
|\|\alpha(\mu)\|-\|\alpha(\mu_n)\||\leq\mu(G)-\mu(K_n).
\]
Step 2 applies to $\mu_n$ for any $n$. Hence Step 3 follows from Step 2.

Step 4: we assume \emph{the support of the measure $\mu$ has positive Haar measure}.  We consider $\eta=\mu*\mu^*$. As the Haar measure of the support of $\mu$ is strictly positive,
\[
	\mbox{support}(\eta)=\overline{(\mbox{support}(\mu))\cdot(\mbox{support}(\mu))^{-1}}
\]
is a neighborhood of the identity according to \cite[page 50]{Weil}. (In order to prove the inclusion $\mbox{support}(\mu*\mu^*)\subset\overline{\mbox{support}(\mu)\cdot\mbox{support}(\mu^*)}$,
we use that $\mu$ is inner regular. The other inclusion is true for any Borel measure on any topological group.)  Again, for any unitary representation $\alpha$ of $G$, $\|\alpha(\eta)\|=\|\alpha(\mu)\|^2$. Hence Step 4  follows from Step 3. 

Step 5: we assume the measure $\mu$ is as in the proposition. We choose a compact neighborhood $K$ of $e\in G$ and define, for any integer $n\in\mathbb N$, the  measure $\mu_n=\mu+\frac{1}{n}{\bold 1}_{K}\cdot\mu_G$. By construction, the support of $\mu_n$ has positive Haar measure. Moreover, $\lim_{n\to\infty}|\mu_n(G)-\mu(G)|=0$. Hence Step 5  follows from Step 4.  

\end{proof}

\subsection{Nets and volume estimates}

The goal of this subsection is to establish counting inequalities that allow us to estimate the cardinal of maximal nets, that is, finite sets of evenly-spaced points, in big subsets (that can be thought as balls). The precise technical statement that will be used later in the proofs is Lemma \ref{lem estimation cardinal des filets}.

\begin{lem}\label{lemme recouvrement fini a droite} Let $G$ be a topological group, $A$ and $B$ two compact subsets, such that $B$ is a neighborhood of $e$. Then, for all $k \in \nn$, there exists a finite subset $S \subset G$ such that
we have, for all $n \in \mathbb{N}^*$, \[B^{k+n-1}A \subset B^n S.\]
\end{lem}

\begin{proof} Let $k \in \nn$. Let $U$ be an open set in $G$ containing $e$ and contained in $B$. Then the compact subset $B^kA$, is covered by the sets $Uh$ for $h \in B^{k}A$. Therefore, there is a finite subset $S$ of $B^{k}A$ such that $B^{k}A \subset US$. Since $US \subset BS$, we therefore have $B^{k}A \subset BS$, hence, for all $n \in \nn^*$, $B^{k+n-1}A \subset B^nS$.
\end{proof}

Note that in the following Lemma, we consider a right-Haar measure.

\begin{lem}\label{croissance geometrique} Let $G$ be a locally compact Hausdorff group, let $\mu_G$ be a right-Haar measure on $G$, let $A,B$ be compact subsets of $G$, such that $B$ is a neighborhood of  $e$. Then, for all $k \in \nn$, there is a constant $c$ such that \[\forall n\in \nn^*,\quad \mu_G(B^{n+k-1}A) \leq c\mu_G(B^n).\]
\end{lem}

\begin{proof} The claim follows immediately from Lemma \ref{lemme recouvrement fini a droite} : it is enough to take $c := \vert S \vert$ where $S$ is a finite subset given by Lemma \ref{lemme recouvrement fini a droite}.
\end{proof}

\begin{defi}[Net] Let $G$ be any group. Let $A,B,N \subset G$. We say that $N$ is an $A$-separated net of $B$ if
\begin{enumerate}
\item $N$ is a finite subset of $B$ ;
\item $\forall n_1,n_2 \in N,\quad n_1A \cap n_2A \neq \emptyset \Rightarrow n_1 = n_2$.
\end{enumerate}
Let us denote by $N(B,A)$ the set of $A$-separated nets in $B$. We endow $N(B,A)$ with the order given by inclusion and denote by $N_{max}(B,A)$ the set of maximal nets (that is, nets that are not strictly contained in a bigger net).
\end{defi}

\begin{lem}\label{lem reunion filet} Let $G$ be any group, and let $A,B \subset G$. We have \[\forall N \in N_{max}(B,A),\quad B \subset NAA^{-1}.\]
\end{lem}

\begin{proof} Let $N \in N_{max}(A,B)$. Let us show that we have $B \subset NAA^{-1}$. We proceed by contradiction. Assume that $b \in B \setminus NAA^{-1}$. Let us show that $N \cup \{b\}$ is an  $A$-separated net of $B$. It is enough to check that for every  $n \in N$, $nA \cap bA = \emptyset$. Let $n \in N$. If $nA \cap bA \neq \emptyset$, there exists $a_1,a_2 \in A$ such that $na_1 = ba_2$. Let $a_1,a_2$ as such. Then $b = na_1a^{-1}_2 \in NAA^{-1}$, which is a contradiction. So $nA \cap bA = \emptyset$. Therefore, $N \cup \{b\}$ is an $A$-separated net of $B$, and this contradicts the maximality of $N$. So we have $B \subset NAA^{-1}$.
\end{proof}

\begin{lem}\label{cardinal filet mesure} Let $G$ be a locally compact Hausdorff group, and $\mu_G$ be a left-Haar measure on $G$. Let $A,B \subset G$ such that $A$  $B$ and $BA$ are measurable. We then have \[\forall N \in N(B,A), \quad \vert N \vert \mu_G(A) \leq \mu_G(BA),\]and\[
\forall N \in N_{max}(B,A), \quad \mu_G(B) \leq \vert N \vert \mu_G(AA^{-1}).\]
\end{lem}

\begin{proof} Let us prove the first item. From the definition of a net, $NA$ is the disjoint union of the $nA$ for $n$ in $N$, so $\mu_G(NA) = \vert N \vert \mu_G(A)$. But since $N\subset B$, $NA \subset BA$, so we get the desired inequality.
Let us now prove the second item. From Lemma \ref{lem reunion filet}, $B \subset NAA^{-1}$. So, we have $\mu_G(B) \leq \mu_G(NAA^{-1}) \leq \mu_G(AA^{-1}) \vert N \vert$.
\end{proof}

\begin{lem}\label{lem estimation cardinal des filets}

Let $G$ be a locally compact Haudorff unimodular group, $A_1,A_2,B,$ compact neighborhoods of $e$ in 
$G$. Then there are constants $c_1,c_2 >0$ such that \[\forall n \geq 3,\ \forall N_1 \in N_{max}(B^n,A_1), \forall N_2 \in N_{max}(B^{n-2},A_2),\quad c_1 \leq \fra{\vert N_2 \vert}{\vert N_1 \vert}\leq c_2.\]
\end{lem}

\begin{proof}
From Lemma \ref{croissance geometrique}, there are strictly positive $k_1,k_2,k_3$ such that
\begin{equation}\label{eq: k_1}
   \forall n\in \mathbb{N},\quad  \mu_G(B^nA_1) \leq k_1\mu_G(B^n),
\end{equation}
\begin{equation}\label{eq: k_2}
   \forall n\geq 3,\quad  \mu_G(B^{n-2}A_2) \leq k_2\mu_G(B^n),
\end{equation}
\begin{equation}\label{eq: k_3}
    \forall n\in \mathbb{N},\quad  \mu_G(B^{n+2}) \leq k_3\mu_G(B^n).
\end{equation}
Let $n \geq 3$, $N_1 \in N_{max}(B^n,A_1)$, $N_2 \in N_{max}(B^{n-2},A_2)$. We have
\begin{equation}\label{eq: majoration N-2n}
    \vert N_2 \vert \stackrel{Lemma\ \ref{cardinal filet mesure}}{\leq} \frac{\mu_G(B^{n-2}A_2)}{\mu_G(A_2)} \stackrel{\eqref{eq: k_2}}{\leq} \fra{k_2}{\mu_G(A_2)}\mu_G(B^n)
\end{equation}
Moreover, 
\begin{equation}\label{eq: minoration N-2n}
    \mu_G(B^{n-2})\frac{1}{\mu_G(A_2A^{-1}_2)} \stackrel{Lemma\ \ref{cardinal filet mesure}}{\leq} \vert N_2 \vert .
\end{equation} 
So, we have
\begin{equation*}
\fra{1}{k_3\mu_G(A_2A^{-1}_2)}\mu_G(B^n) \stackrel{\eqref{eq: k_3}}{\leq} \mu_G(B^{n-2}) \fra{1}{\mu_G(A_2A^{-1}_2)} \stackrel{\eqref{eq: minoration N-2n}}{\leq} \vert N_2 \vert \stackrel{\eqref{eq: majoration N-2n}}{\leq} \fra{k_2}{\mu_G(A_2)}\mu_G(B^n).    
\end{equation*}
On the other hand, 
\[\vert N_1 \vert \stackrel{Lemma\   \ref{cardinal filet mesure}}{\leq} \frac{\mu_G(B^nA_1)}{\mu_G(A_1)},\] 
so 
\[\vert N_1 \vert \stackrel{\eqref{eq: k_1}}{\leq} \frac{k_1}{\mu_G(A_1)}\mu_G(B^n).\] 
Moreover, \[\mu_G(B^n)\frac{1}{\mu_G(A_1A^{-1}_1)}\stackrel{Lemma\  \ref{lem estimation cardinal des filets}}{\leq} \vert N_1 \vert.\] So we have \[\mu_G(B^n)\frac{1}{\mu_G(A_1A^{-1}_1)} \leq \vert N_1 \vert \leq \frac{k_1}{\mu_G(A_1)}\mu_G(B^n).\]We therefore obtain \[\frac{\mu_G(A_1)}{k_3k_1\mu_G(A_2A^{-1}_2)} \leq \frac{\vert N_2 \vert}{\vert N_1 \vert} \leq \frac{k_2 \mu_G(A_1A^{-1}_1)}{\mu_G(A_2)}.\]
\end{proof}

\subsection{Proof of Theorem \ref{thm: orbites negligeables} and Corollary \ref{coro: actions sur reseaux}}\label{proofs of the corollaries}

In order to prove Theorem \ref{thm: orbites negligeables}, it is obviously enough to prove the following proposition: 

\begin{prop}\label{prop bons ouverts} 

Let $G$ be a locally compact Hausdorff unimodular group. Let $X$ be a Hausdorff topological space, $\nu$ an atomless regular Borel measure on $X$. Consider a continuous measure-preserving action $G \curvearrowright X$. We assume that there exists $x_0$ in the support of $\nu$ such that $\Stab_G(x_0)$ is compact and for any compact subset $K$ of $G$, $\nu(Kx_0) = 0$.
Then for each compact symmetric neighborhood $S$ of $e$ in $G$, there exists a compact, symmetric neighborhood $F$ of $e$, and a sequence of neighborhoods of $x_0$ that is of $(S,F)$-moderate growth. 
\end{prop}

\begin{rema} If $G$ is $\sigma$-compact, then $\nu(Kx) = 0$ for every compact $K$ if and only if the orbit of $x_0$ is a null-set.
\end{rema}

We will need the following lemma.

\begin{lem} \label{prop retours lointains}
Let $G$ be a topological group, $X$ a Hausdorff topological space. Let $G \curvearrowright X$ be a continuous action (that is, the map $G\times X \rightarrow X$ that defines it is continuous).

Let $x_0 \in X$, let $K$ be a compact subset of $G$, and $U$ a neighborhood of $e$ in $G$. Then there exists a neighborhood $V$ of $x_0$ in $X$ such that \[\forall g \in K\setminus \Stab_G(x_0)U,\quad gV \cap V = \emptyset.\]
\end{lem}

\begin{proof} First of all, let us notice that $\Stab_G(x_0)U$ is open, and therefore $K\setminus\Stab_G(x_0)U$ is compact. 

Let $E$ be the set of triples $(g,W,V)$ such that \begin{itemize}
\item $g \in K \setminus \Stab_G(x_0)U$;
\item $W$ is a neighborhood of $g$ in $G$;
\item $V$ is a neighborhood of $x_0$ in $X$;
\item $\forall w\in W,\quad wV \cap V = \emptyset$.
\end{itemize}

For all $g \in K \setminus \Stab_G(x_0)U$, there exists $W,V$ such that $(g,W,V) \in E$. Indeed, $g \not \in \Stab_G(x_0)$, so $gx_0 \neq x_0$. Let $Y$ be a closed neighborhood of $x_0$ that does not contain $gx_0$. By continuity, the set of  $(h,x)$ such that $hx \not \in Y$ is an open subset of  $G\times X$; moreover, it contains $(g,x_0)$, so there are neighborhoods $W,Z$ of $g$ and $x_0$ such that $WZ \subset Y^c$. Let $V := Z\cap Y$. Then $WV \subset Y^c \subset V^c$, so $WV \cap V = \emptyset$. 

Therefore, the set of second coordinates of $E$ covers $K \setminus \Stab_G(x_0)U$. By compacity, there exists $n \in \nn^*$, $g_1,\cdots,g_n$, $W_1,\cdots,W_n$, $V_1,\cdots,V_n$ such that $K \setminus \Stab_G(x_0)U \subset \bigcup_i W_i$. Let us define $V := \bigcap_i V_i$.  This subset $V$ satisfies the needed requirements: it is a neighborhood of $x_0$, and if $g \in K \setminus \Stab_G(x_0)U$, then there exists $i$ such that $g \in W_i$. Therefore $gV_i \cap V_i = \emptyset$, so $gV\cap V =\emptyset$.
\end{proof}

\begin{proof}[Proof of Proposition \ref{prop bons ouverts}] Let $S$ be a compact, symmetric neighborhood of $G$. We would like to find a compact symmetric neighborhood $F$ of $e$ and a sequence $(B_n)_{n \geq 3}$ that has $(S,F)$-moderate growth. Let $n \geq 3$. 

First of all, since $\nu(S^{2n}x_0) = 0$, and since $\nu$ is regular, there exists an open subset $U$ of $X$ containing $S^{2n}x_0$, that has measure strictly smaller than $\frac{1}{2}$. 

We claim that there exists a neighborhood $W_n$ of $x_0$ such that $S^{2n}W_n \subset U$. Indeed, by continuity of the action, for all $g \in S^{2n}$, there exists a neighborhood $D_g$ of $g$ and a neighborhood $E_g$ of $x_0$ such that $D_gE_g \subset U$. By compactness of $S^{2n}$, there exists a finite subset $I$ of $S^{2n}$ such that $S^{2n} \subset \bigcup_{i \in I} D_i$. Let us now define $W_n := \bigcap_{i \in I} E_i$. We have \[S^{2n} W_n \subset \bigcup_{i \in I} D_iW_n \subset \bigcup_{i \in I} D_iE_i \subset U,\]which gives the claim.

Now, taking $F = S$, according to Lemma \ref{prop retours lointains}, there exists a neighborhood $V_n$ of $x_0$ such that 

\begin{equation}\label{eq: xx}
\forall g \in S^{2n-2} \setminus \Stab_G(x_0)S, \quad gV_n \cap V_n = \emptyset;
\end{equation}

moreover, up to taking a smaller neighborhood instead of $V_n$, we can assume that $V_n$ is contained in $W_n$.

Let us then define \[B_n := S^nV_n.\]We now check that the sequence $(B_n)_{n \geq }$ fulfills the three requirements of $(S,F)$-moderate growth.

First of all, since $V_n$ is a neighborhood of $x_0$ and since $x_0$ is in the support of $\nu$, $\nu(B_n) > 0$, so the first condition is verified. Moreover, by construction, $S^nB_n = S^{2n}V_n \subset U$; and $\nu(U) < \frac{1}{2}$ so the second condition is satisfied.

We will now estimate, for all $g \in F$, \[\fra{\nu(gB_n \cap B_n)}{\nu(B_n)}\]by considering nets.

Let $N_{1} \in N_{max}(S^{n-2},S\Stab_G(x_0)S^2 )$. We have $N_{1}F \subset F^{n-1}$, so 
\begin{equation}\label{eq: xxx}
\nu(N_{1}FV_n) \leq \nu(F^{n-1}V_n).
\end{equation}
Moreover, we have \[N_{1}FV_n = \bigsqcup_{g \in N_{1}} gFV_n;\] indeed, on the one hand, let $g_1,g_2 \in N_{1},\ f_1,f_2 \in F,\ v_1,v_2 \in V_n$ such that $g_1f_1v_1 = g_2f_2v_2$. We then have $f^{-1}_2g^{-1}_2g_1f_1v_1 = v_2$. On the other hand, $f^{-1}_2g^{-1}_2g_1f_1 \in F^{2n-2}$, so, by condition \eqref{eq: xx} on $V_n$, we have $f^{-1}_2g^{-1}_2g_1f_1 \in \Stab_G(x_0)F$. So, $g_1 \in g_2F\Stab_G(x_0)F^2$, and this is only possible if $g_1 = g_2$.

We deduce from this \[\nu(N_{1}FV_n) = \vert N_{1} \vert \nu(FV_n),\]so \[\vert N_{1} \vert \nu(FV_n) \stackrel{\eqref{eq: xxx}}{\leq} \nu(F^{n-1}V_n).\]

Now, for any $g \in F$, we have $F^{n-1}V_n \subset gF^nV_n \cap F^nV_n = gB_n \cap B_n$, so \[\nu(gB_n \cap B_n) \geq \vert N_{1} \vert \nu(FV_n).\]

Finally, let $C$ be a compact neighborhood of $e$ in $G$ such that $CC^{-1} \subset F$. Let $N_2 \in N_{max}(F^n,C)$. Then \[B_n = F^nV_n \stackrel{Lemma\ \ref{lem reunion filet}}{\subset} N_2CC^{-1}V_n\] so \[\nu(B_n) \leq \vert N_2 \vert \nu(FV_n).\]We therefore have, for all $n \geq 3$, \[\fra{\nu(gB_n \cap B_n)}{\nu(B_n)} \geq \fra{\vert N_{1} \vert}{\vert N_2\vert},\]so \[\left(\fra{\nu(gB_n \cap B_n)}{\nu(B_n)}\right)^{\frac{1}{n}} \geq \left(\fra{\vert N_{1} \vert}{\vert N_2 \vert}\right)^{\frac{1}{n}}.\]The third requirement follows now from Lemma \ref{lem estimation cardinal des filets}.
\end{proof}

We now add a few words to show that Corollary \ref{coro: actions sur reseaux} follows in a straightforward way from Theorem \ref{thm: orbites negligeables}.

\begin{proof}[Proof of Corollary \ref{coro: actions sur reseaux}] It is immediate to check the hypotheses of Theorem \ref{thm: orbites negligeables} are satisfied: let $G$ be a locally compact group, $H$ a closed subgroup of $G$ such that $\mu_G(H) = 0$, and $\Gamma$ a lattice in $G$ such that $H\cap \Gamma = \{e\}$.
The action of $H$ on $G/\Gamma$ by left-translations is continuous, its orbits are null-sets, and the stabilizer  $H \cap \Gamma$ of the coset $e\Gamma$ is finite, hence compact.
\end{proof}

\section{Examples and counterexamples}

\subsection{Margulis' counterexample}
\label{contreex}

This section, provides many examples showing that Theorem \cite[Thm 3.6]{PinPitExactRate}  cannot extend to the case where $G$ is locally compact without making further hypotheses.

\begin{exe} Consider the action by left-translation of $\mathbb{R}$ on $\mathbb{R}/\mathbb{Z}$. Let, for $t \in \rr$, $\pi_0(t)$ be the operator that sends every zero-integral $\phi \in L^2(\rr/\zz,Leb)$ to the map $s \mapsto \phi(s-t)$. Let $\mu$ be the uniform probability measure on $[0,1]$. Then, obviously, $\pi_0(\mu)$ is the zero operator, whereas $\lambda_\rr(\mu)$ is not.
\end{exe}

In fact, we have the following generalization, that follows from the discussion that one can find in pages $107-112$ of the book \cite{MARGU}. We give a sketch of the proof for the convenience of the reader.

\begin{prop}[Margulis]

Let $G$ be a locally compact Hausdorff $\sigma$-compact group, and let $H$ be a closed cocompact subgroup of $G$. We consider the action by left-multiplication of $G$ on $G/H$ and we suppose that there is a $G$-invariant probability measure $\nu$ on $G/H$. Let $\pi$ denote the associated Koopman representation, and $\pi_0$ the sub-representation on the subspace of zero integral. Then there is a continuous, compactly-supported, nonnegative function $f$ on $G$ that has integral $1$, such that \[\Vert \pi_0(f) \Vert < 1.\]
\end{prop}

\begin{proof}[Proof sketch]

Assume, by contradiction, that for every continuous, compactly-supported, nonnegative function $f$ on $G$ that has integral $1$, we have \[\Vert \pi_0(f) \Vert = 1.\]

From \cite[(1.3) Proposition, p. 109]{MARGU}, there exists a sequence $(p_i)_{i \in \nn}$ that is asymptotically $\pi_0$-invariant, that is, each $p_i$ is an element of $L^2_0(G/H)$, has norm $1$, and we have, for every compact $K$ in $G$, \[\lim_{i \to \infty} \sup_{g \in K} \Vert \pi_0(g) p_i - p_i \Vert_2 = 0.\]

Let $(p_i)_{i \in \nn}$ be such a sequence. First of all, we have\[\begin{array}{rcl}
\displaystyle\limsup_{i \to \infty} \Vert \pi_0(f) p_i - p_i \Vert_2 &\leq &\displaystyle\limsup_{i \to \infty} \sup_{g \in \supp(f)} \Vert \pi_0(g) p_i - p_i \Vert_2\\
&= &0\\
\end{array}\]Let $K$ be a compact subset of $G$. We then have \[\begin{array}{rcl}
\displaystyle\limsup_{i \to \infty}\  \sup_{g \in K} \Vert \pi_0(g) \pi_0(f) p_i - \pi_0(f) p_i \Vert_2 &\leq &\displaystyle \limsup_{i \to \infty} \ \sup_{g \in K} \left( \Vert \pi_0(g)\pi_0(f) p_i - \pi_0(g) p_i \Vert_2\right.\\
&&\quad \displaystyle + \left. \Vert \pi_0(g) p_i - p_i\Vert_2 + \Vert p_i - \pi_0(f)p_i\Vert_2\right)\\
&= &\displaystyle \limsup_{i \to \infty} \ \sup_{g \in K} \left( 2\Vert \pi_0(f) p_i - p_i \Vert_2\right.\\
&&\quad + \left.\Vert \pi_0(g) p_i - p_i\Vert_2\right)\\
&= &0\\
\end{array}\]so the sequence\[\left(\fra{\pi_0(f)p_i}{\Vert \pi_0(f)p_i \Vert_2}\right)_{i \in \nn}\]is also asymptotically $\pi_0$-invariant. Let us denote, for all $i \in \nn$, $q_i$ the composition of \[\fra{\pi_0(f)p_i}{\Vert \pi_0(f)p_i\Vert_2}\] with the canonical surjection $G \rightarrow G/H$. So, from \cite[(1.7) Lemma, p. 110]{MARGU}, the sequence $(q_i)_{i \in \nn}$ is equicontinuous and uniformly bounded on $G$. From Ascoli's theorem, up to extraction, one can assume that it converges uniformly on compact subsets to a continuous, $H$-invariant on the right function (since each term of the sequence is $H$-invariant on the right). This limit gives, on the quotient $G/H$, a function that we denote by $p$. Since $G/H$ is compact, we have the uniform convergence \[\fra{\pi_0(f) p_i}{\Vert \pi_0(f)p_i \Vert_2} \rightarrow p.\]We deduce from all this \begin{itemize}[label=\textbullet]
\item that $p$ has zero integral (since it is the limit of a sequence of zero-integral functions);
\item that $p$ has $L^2$-norm $1$ (as a uniform limit of uniformly bounded functions of $L^2$-norm $1$);
\item that $p$ is $\pi_0$-invariant (as a uniform limit of $\pi_0$-invariant functions) and therefore is constant.
\end{itemize}

So, $p$ is identically $0$, but has $L^2$-norm $1$. This is a contradiction.
\end{proof}

\begin{rema}
In particular, if $G$ is a locally compact, $\sigma$-compact amenable group, if $H$ is a cocompact closed subgroup of $G$ such that $G \curvearrowright G/H$ has an invariant probability measure, and if $\pi_0$ denotes the sub-representation of the Koopman representation on the subspace of zero-integral functions, then there is a continuous, positive, symmetric function $h$ of integral $1$ 
such that 
\[ \Vert \pi_0(h)\Vert<\Vert \lambda_G(h) \Vert \]
holds (that is, Inequality (1) doesn't hold).
To prove it, we first recall the following notation from \cite[p. 282]{DIXMIER}: we  denote by $\Delta$ the modular function of $G$, and for all continuous, compactly-supported complex-valued $f$ on $G$, we denote by $f^*$ the map defined by \[\forall g \in G,\quad f^*(g) = \Delta(g^{-1}) \overline{f(g^{-1})}.\]This is the involution of the involutive algebra $C_c(G)$. Now, take $f$ as in Margulis' theorem, and let \[h := \frac{1}{\int_G f^* * f\diff \mu_G}f^* * f.\] Then $h$ satisfies the aforementioned properties. Moreover, from Kesten's theorem, since $G$ is amenable, the operator $\lambda(h)$ has norm $1$; and, according to Margulis' theorem, $\pi_0(h)$ has norm strictly less than $1$.
\end{rema}

\subsection{An action on a finite-volume homogeneous space with optimal discrepancy}
\label{sect: exemple de Zimmer}

In this section, we give an example of an action of a non-amenable locally compact group that acts continuously on a measured topological space by preserving the measure, where the discrepancy inequality is an equality.

Let $G := \sln_3(\mathbb{R}) \times \sln_3(\mathbb{R})$, and consider the subgroup \[\left\{\begin{pmatrix}
A &v\\
0 &1\\
\end{pmatrix} \ \vert \ A \in \sln_2(\mathbb{R}),\ v \in \mathbb{R}^2\right\} \simeq \mathbb{R}^2 \rtimes \sln_2(\mathbb{R})\] of $\sln_3(\mathbb{R})$ and consider its embedding $H$ in the first factor of $G$.

Consider now the unique (unital) ring morphism $\sigma : \mathbb{Z}[\sqrt{2}] \rightarrow \mathbb{C}$ sending $\sqrt{2}$ to $-\sqrt{2}$. For a matrix $A := (a_{ij})_{i,j =1,2,3} \in \sln_3(\mathbb{Z}[\sqrt{2}])$, let us denote also by $\sigma(A)$ the matrix $(\sigma(a_{ij}))_{i,j =1,2,3}$. Since $\sigma$ is a ring morphism, \[\sigma(A) \in \sln_3(\mathbb{Z}[\sqrt{2}]).\]

Consider now the map \[\begin{array}{rcl}
j : \sln_3(\mathbb{Z}[\sqrt{2}]) &\rightarrow &G\\
A &\mapsto &(A,\sigma(A)).\\
\end{array}\]

Then $j$ is an injective group morphism, and $\Gamma := j(\sln_3(\mathbb{Z}[\sqrt{2}]))$ is an irreducible lattice in $G$ (see \cite[Example 2.2.5, p.18]{ZIMMER}). Obviously, as $H$ is contained in the first factor of $G$, it follows that $\Gamma \cap H = \{e\}$.

\begin{prop}\label{exemple optimal}

Consider the action $H \curvearrowright (G/\Gamma,\mu_{G/\Gamma})$ and the sub-representation $\pi_0$ of the Koopman representation of $H$ on $L^2_0(G/\Gamma,\mu_{G/\Gamma})$. 

Then, for every positive finite Borel regular measure $\mu$ on $H$, the inequality in Corollary \ref{coro: actions sur reseaux} is optimal, that is, we have \[\Vert \lambda_H(\mu) \Vert = \|\pi_0(\mu)\|.\]
\end{prop}

\begin{proof}
By construction, we can apply Corollary \ref{coro: actions sur reseaux}, so we have, $\|\pi_0(\mu)\|\geq \Vert \lambda_H(\mu) \Vert$.
For the other inequality, notice that since $\Gamma$ is irreducible, the action $G \curvearrowright G/\Gamma$ is mixing (see for example \cite[Corollary 3]{pinohowemoore}), and therefore the restriction of this action to any non-compact subgroup is ergodic.
In particular, the restriction of $\pi_0$ to the embedded copy of $\mathbb{R}^2$ in $H$ has no nontrivial invariant vectors. So, according to \cite[Theorem 7.3.9, p.146]{ZIMMER}\footnote{The reader willing to read in Zimmer's book the result we quote should be warned that in the literature, the symbol $\prec$ can denote two slightly different notions of weak containment (see \cite[Remark F.1.2 (ix), p. 397]{BHV} for a detailed account of the differences between the two).}, we have the weak containment \[\pi_0 \prec \lambda_H.\]
According to \cite[Remark 8.B.6 (2) page 244]{BekDelaHa}) this implies the inequality $\|\pi_0(\mu)\|\leq \Vert \lambda_H(\mu) \Vert$.
\end{proof}

\end{document}